\documentclass[10pt]{amsart}

\usepackage{amssymb, amsmath} 
\usepackage{enumerate,verbatim}
\usepackage[all]{xy}
\usepackage{graphicx}
\usepackage{pstricks}

\newtheorem{theorem}{Theorem}[section]
\newtheorem*{theorem-nn}{Theorem}
\newtheorem{proposition}[theorem]{Proposition}
\newtheorem{lemma}[theorem]{Lemma}
\newtheorem{corollary}[theorem]{Corollary}

\newtheorem*{conjecture-nn}{Conjecture} \theoremstyle{definition}
\newtheorem{definition}[theorem]{Definition}
\newtheorem{remark}[theorem]{Remark}

\newcommand{\Fraisse}{Fra\"iss\'e } 
\newcommand{\cK}{\mathcal{K}}
\newcommand{\cM}{\mathcal{M}} 

\newcommand{\cKp}{\mathcal{K}_p} 
\newcommand{\cMp}{\mathcal{M}_p}

\newcommand{\bbQU}{\mathbb{QU}}
\newcommand{\oUr}{\mathbb{QU}_{\prec}}

\newcommand{\ra}{\rightarrow}
 
\newcommand{\id}{1}

\newcommand{\aut}[1]{\mathop{\mathrm{Aut}}(#1)}
\newcommand{\homeo}[1]{\mathop{\mathrm{Homeo}}(#1)}
\newcommand{\homeop}[1]{\mathop{\mathrm{Homeo}}^{+}(#1)}
\newcommand{\orb}[2]{\mathop{\mathrm{Orb_{#1}}}(#2)}

\newcommand{\den}{\mathcal{D}} 

\newcommand{\iso}[1]{\mathop{\mathrm{Iso}}(#1)}
\newcommand{\isom}[2]{\mathop{\mathrm{Iso}_{#1}}(#2)}
\newcommand{\dom}[1]{\mathop{\mathrm{dom}}(#1)}
\newcommand{\ran}[1]{\mathop{\mathrm{ran}}(#1)}
\newcommand{\es}{\emptyset} 
\newcommand{\eps}{\varepsilon}
\newcommand{\diam}{diam} 
\newcommand{\fix}[1]{\mathop{\mathrm{F}}(#1)}
\newcommand{\per}[1]{\mathop{\mathrm{Z}}(#1)}

\newcommand{\essen}[1]{\mathop{\mathrm{Ess(#1)}}}

\newcommand{\sign}{\mathop{\mathrm{sign}}}

\newcommand{\bA}{\mathbf{A}} 
\newcommand{\bB}{\mathbf{B}}
\newcommand{\bC}{\mathbf{C}} 
\newcommand{\bD}{\mathbf{D}}
\newcommand{\bE}{\mathbf{E}} 

\newcommand{\bbG}{\mathbb{G}} 
\newcommand{\bbK}{\mathbb{K}}
 
\newcommand{\bbN}{\mathbb{N}}
\newcommand{\bbQ}{\mathbb{Q}} 
\newcommand{\bbU}{\mathbb{U}}

\newcommand{\solecki}{Solecki} 
\newcommand{\ssolecki}{S\l{}awomir
  Solecki }

\newcommand{\N}{\mathbb{N}}
\newcommand{\Q}{\mathbb{Q}} 

\newcommand{\Z}{\mathbb{Z}}

\newcommand{\lefint}{L} 
\newcommand{\rigint}{R}

\newcommand{\tsn}{E_{TS}^n}

\author{Konstantin Slutsky}

\title{Non-Genericity Phenomenon in Some Ordered Fra\"iss\'e classes}
\begin{document}
\maketitle

\begin{abstract}
  We show that every two-dimensional class of topological similarity,
  and hence every diagonal conjugacy class of pairs, is meager in the
  group of order preserving bijections of the rationals and in the
  group of automorphisms of the randomly ordered rational Urysohn
  space.
\end{abstract}

\section{Introduction}
\label{sec:intro}

The size of conjugacy classes (in the topological sense: dense,
meager, comeager, etc.) in the groups of automorphisms of the \Fraisse
limits has recently become an active area of research. This is
partially due to the newly revealed connections between combinatorial
properties of the \Fraisse classes and algebraic, topological, and
dynamical properties of the groups of automorphisms of their
limits. One of the most astonishing links was established by Kechris,
Pestov and Todorcevic in \cite{KPT} and displays close relationship
between the Ramsey theory (a purely combinatorial area) and the extreme
amenability (a classical dynamical notion).

Another reason for the interest in the size of conjugacy classes of
Polish groups in general, and groups of automorphisms of \Fraisse
limits in particular, comes from the special importance of some
concrete groups, e.g., the group of automorphisms of the countable
atomless Boolean algebra (which is isomorphic via the Stone's theorem
to the group of homeomorphisms of the Cantor space), the group of
isometries of the rational Urysohn space, and the group of order
preserving bijections of the rationals. The reader may consult
\cite{KR} for details, examples, and a deep structural theory for the
groups with large conjugacy classes.

In \cite{T} J. K. Truss looked at different possible notions of
genericity of conjugacy classes and discussed advantages of each. In
general, conjugacy classes are objects that are difficult to
understand, and the relation of conjugation may sometimes be very
complicated (complete analytic). Motivated by the work of Truss, in
this paper we look at a coarser equivalence relation than conjugation,
namely at classes of topological similarity. They are much easier to
work with and can be used to prove the {\it non}-genericity in some
cases.

Let us
recall the definition from \cite{R}.

\begin{definition}
  Let $G$ be a topological group, an $n$-tuple $(g_1, \ldots, g_n) \in
  G^n$ is said to be {\it topologically similar\/} to an $n$-tuple
  $(f_1, \ldots, f_n) \in G^n$ if the map $F$ sending $g_i \mapsto
  f_i$ extends (necessarily uniquely) to a bi-continuous isomorphism
  between the groups generated by these tuples
  \[ F : \langle g_1, \ldots, g_n \rangle \ra \langle f_1, \ldots, f_n
  \rangle.\]We denote this relation by $\tsn$.
\end{definition}

There is another natural relation on the $n$-tuples in $G$, namely the
relation of diagonal conjugation, i.e, $(g_1, \ldots, g_n)$ is
conjugate to $(f_1, \ldots, f_n)$ if there is some $\alpha \in G$ such
that
\[(\alpha g_1 \alpha^{-1}, \ldots, \alpha g_n \alpha^{-1}) = (f_1,
\ldots, f_n).\] More generally, if $G$ is a topological subgroup of
$H$, then one can restrict the conjugacy relation in $H$ to $G$, i.e,
say that an $n$-tuple $(g_1, \ldots, g_n) \in G^n$ conjugates in $H$
to an $n$-tuple $(f_1, \ldots, f_n) \in G^n$ if there is some $\beta
\in H$ such that
\[(\beta g_1 \beta^{-1}, \ldots, \beta g_n \beta^{-1}) = (f_1, \ldots,
f_n).\] It is easy to see that this is an equivalence relation on
$n$-tuples of $G$ and we denote it by $E_H^n$ (this relation also
depends on $G$ and on the embedding of $G$ into $H$, but this
information is usually clear from the context). The following proposition
is obvious.
\begin{proposition}
  Let $G$ be a topological group and $n \in \N$, then
  \begin{enumerate}[(i)]
  \item relation of topological similarity is an equivalence relation;
  \item let $H$ be any topological group such that $G \le H$ is a
    topological subgroup of it, then $E_H^n$ is finer (not necessarily
    strictly) than $\tsn$. In particular, $E_G^n$ is finer than
    $\tsn$.
  \end{enumerate}
\end{proposition}

Equivalence classes of topological similarity on $n$-tuples are called
{\it $n$-dimensional similarity classes}, in particular
two-dimensional similarity class is a set of pairs.

Since conjugacy classes refine classes of topological similarity, if
one wants to prove meagerness of the former, it suffices to prove
meagerness of the latter (it suffices, but may be impossible,
meagerness of conjugacy classes does not implies meagerness of classes
of topological similarity). This sometimes turns out to be an easier
task. For example, Rosendal in \cite{R} developed this idea to find a
simple proof of the del Junco's result, that each conjugacy class of
measure preserving automorphisms of the standard Lebesgue space is
meager.

Hodkinson showed (see \cite{T} for the details) that in the group of order
preserving automorphisms of the rationals all conjugacy classes of
pairs are meager (though it is known that there is a comeager
one-dimensional conjugacy class in $\aut{\Q}$). In this paper we
strengthen this result and show that all two-dimensional classes of
topological similarity are meager in this group.

The paper is organized as follows. In the second section we give a
brief introduction to the theory of \Fraisse classes, which is the
right context for the technique, developed in this paper. The third
section is devoted to the strengthening of the Truss's result on
classes of diagonal conjugation in the rationals, in the fourth and
fifth sections analogous theorem is proved for the group of order
preserving isometries of the randomly ordered rational Urysohn space.

{\bf Acknowledgment.} The author wants to thank Christian Rosendal
for advising him during the process of writing this paper, for posing
questions, that are covered here, and for helpful and very inspiring
communication and comments. The author also thanks \ssolecki and Ward
Henson for helpful discussions.

\section{Brief introduction to \Fraisse classes}
\label{sec:introFr}

In this section we give a short introduction to the theory of \Fraisse
classes. In the next two sections we deal with two examples of them, so it
is useful to keep in mind this more general setting. The classical text on
\Fraisse classes is a beautiful book by Hodges \cite{H}. 

Let $L$ be a relational first order language. We use the standard
notation: solid arrows correspond to ``for all'' quantifiers, and
dashed arrows represent maps that ``exist''.

\begin{definition}
Let $\bA$ and $\bB$ be two $L$-structures. A map $f : \bA \to \bB$ is
called a {\it strong homomorphism} if for any $\mathcal{R} \in L$ of
arity $m$ and for any $x_1, \ldots, x_m \in \bA$
\[\mathcal{R}^{\bA}x_1 \ldots x_m \iff \mathcal{R}^{\bB}f(x_1)\ldots f(x_m).\]
The map is a {\it strong embedding} if it is an injective strong
homomorphism.
\end{definition}

\begin{definition}
  Let $\cK$ be a class of finite $L$-structures. For $L$-structures $\bA$ and
  $\bB$ by $\bA \le \bB$ we mean ``$\bA$ strongly embeds into
  $\bB$''. $\cK$ is called a \Fraisse class if the following
  properties hold:
  \begin{enumerate}[(i)]
  \item[(HP)] If $\bA \le \bB$ and $\bB \in \cK$ then $\bA \in \cK$;
  \item[(JEP)] For $\bA \in \cK$ and $\bB \in \cK$ there is some $\bC
    \in \cK$ such that $\bA \le \bC$ and $\bB \le \bC$;
  \item[(AP)] For $\bA \in \cK$, $\bB \in \cK$, $\bC \in \cK$ and
    embeddings $i : \bA \ra \bB$, $j : \bA \ra \bC$ there is $\bD \in
    \cK$ and embeddings $k : \bB \ra \bD$, $l : \bC \ra \bD$ such that
    $k \circ i = l \circ j$, i.e., the following diagram commutes

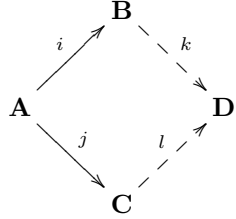
\[\xymatrix{    & \bB \ar@{-->}[dr]^{k} &    \\
          \bA \ar[ur]^{i} \ar[dr]^{j} &     & \bD\\
              & \bC \ar@{-->}[ur]^{l} &     }\]

  \item[(Inf)] $\cK$ contains structures of arbitrarily high finite
    cardinality and has up to isomorphism only countably many
    structures.
  \end{enumerate}
\end{definition}


Basic examples of \Fraisse classes are: finite sets, finite linear
orders, finite graphs and finite metric spaces with rational distances
(to satisfy (Inf) condition). With a \Fraisse class
$\cK$ one can associate its \Fraisse limit (which is unique up to an
isomorphism). 


\begin{definition}
  Countably infinite structure $\bbK$ is called a {\it \Fraisse
    limit\/} of the class $\cK$ if the following holds:
  \begin{enumerate}[(i)]
  \item Finite substructures of $\bbK$ up to isomorphism are exactly
    the elements of $\cK$;
  \item $\bbK$ is ultrahomogeneous (that is any isomorphism between
    finite substructures of $\bbK$ extends to a full automorphism of
    $\bbK$).
  \end{enumerate}
\end{definition}

\Fraisse limits of the above \Fraisse classes are: $\bbN$ ---
countably infinite set, $\bbQ$ --- dense linear ordering without
endpoints, $\bbG$ --- random graph, $\Q\bbU$ --- rational Urysohn space.

In the next section we deal with the simplest linearly ordered
\Fraisse class: with the rationals $\Q$. The fourth and fifth sections
are devoted to the case of the linearly ordered rational Urysohn
space.

\section{Topological Similarity Classes in the Groups $\aut{\Q}$ and
  $\homeop{[0,1]}$}
\label{sec:autq}

Let $\Q$ denote rational numbers viewed as a linearly ordered set. By
an open interval $I = (a,b) \subset \Q$ we mean the set of rational
numbers $\{c : a < c < b\} \subset \Q$. A closed interval $[a,b]$ also
includes endpoints $a$ and $b$. If $I$ is a bounded interval (open or
closed) $\lefint(I)$ will denote its left endpoint and $\rigint(I)$
will be its right endpoint. If $\bA \subset \Q$ is a finite subset,
$\min(\bA)$ and $\max(\bA)$ will denote its minimal and maximal
elements respectively.  

Let $G$ denote the group $\aut{\Q}$ of order
preserving bijections of the rationals.

\begin{definition}
  A partial isomorphism of $\Q$ is an order preserving bijection $p$
  between {\it finite\/} subsets $\bA$ and $\bB$ of $\Q$.
\end{definition}
It is a basic property of rationals (and, as mentioned earlier, of a
\Fraisse limit in general) that each partial isomorphism can be
extended (certainly, not uniquely) to a full automorphism.

Letters $p$ and $q$ (with possible sub- or superscripts) will denote
partial isomorphisms; let $\dom{p}$ be the domain of $p$, and
$\ran{p}$ be its range. If $I \subseteq \Q$ then $p|_{I}$
denotes restriction of $p$ on $I \cap \dom{p}$; $\fix{p}$ will be the set of
fixed points in the domain of $p$, i.e,
\[\fix{p} = \{c \in \dom{p} : p(c) = c\}.\]

First we recall that $G$ is a Polish group (i.e., a separable
completely metrizable topological group) in the topology given by the
basic open sets
\[U(p) = \{g \in G : \textrm{$g$ extends $p$}\},\] where $p$ is a
partial isomorphism of $\Q$. Note that if $p$ and $q$ are two partial
isomorphisms and $q$ extends $p$ then $U(q) \subseteq U(p)$; we will
use this observation frequently. We denote the identity element of $G$
by $\id$. 

We use words generic and comeager as synonyms. For example, a property
is generic in the group $G$ if the set of elements with this property
is comeager in $G$.

Let $F(s,t)$ denote the free group on two generators: $s$ and $t$;
elements of $F(s,t)$ are reduced words on the alphabet $\{s, t,
s^{-1}, t^{-1}\}$. Every element $w \in
F(s,t)$ has certain length associated to it, namely the length of the
reduced word $w$. This length is denoted by
$|w|$. If $u, v \in F(s,t)$ are words, we say that the word $uv \in
F(s,t)$ is reduced if $|uv| = |u| + |v|$, that is there is no
cancellation between $u$ and $v$.

If $w \in F(s,t)$ is a reduced word, $w=t^{n_k}s^{m_k} \cdots
t^{n_1}s^{m_1}$, and $p$, $q$ are partial isomorphisms, then we can
define a partial isomorphisms $w(p,q)$ by $w(p,q)(c) = q^{n_k}p^{m_k}
\cdots q^{n_1}p^{m_1}(c)$, whenever the right-hand side is defined. The
{\it orbit\/} of $c$ under $w(p,q)$ is by definition
\begin{multline*}
  \orb{{\it w(p,q)}}{c} = \cup_{l=1}^k \{p^{i\sign(m_l)} q^{n_{l-1}}
  \cdots p^{m_1}(c),\
  q^{j \sign(n_l)} p^{m_l} \cdots p^{m_1}(c) : \\
  i=0, \ldots, |m_l|,\ j=0, \ldots, |n_l|\}.
\end{multline*}

We say that a word $w$ starts from the word $v$ if $w$ can be written
as $w = vu$ for some word $u$, where $vu$ is reduced. Similarly, we
say that $w$ ends in $v$ if there is a word $u$ such that $w = uv$,
where $uv$ is reduced. On the one hand this is consistent with
intuitive understanding of this notions for, say, left-to-right
languages. On the other hand, we consider left actions, and then {\it
  the end of the word acts first}, i.e., if $w=st$ then $w(p,q)(c) =
p(q(c))$. This may be a bit confusing, we apologize for that and
emphasize this possible confusion.



\begin{definition}
  Let $p$ be a partial isomorphism of $\Q$. An interval $(a,b) \subset
  \Q$ is called $p$-{\it increasing\/} if $a, b \in \dom{p}$, $p(a) =
  a$, $p(b) = b$ and $p(c) > c$ for any $c \in \dom{p} \cap
  (a,b)$. Definition of $p$-{\it decreasing\/} interval is
  analogous. Note that if $[a,b] \cap \dom{p} = \{a, b\}$ and $p(a) =
  a$, $p(b) = b$ then the interval $(a,b)$ is both $p$-increasing and
  $p$-decreasing.  An interval is $p$-{\it monotone\/} if it is either
  $p$-increasing or $p$-decreasing.
\end{definition}

\begin{definition}
  Let $p$ be a partial isomorphism. Let $\dom{p} = \{a_0, \ldots,
  a_n\}$ and assume that $a_0 < \ldots < a_n$. We say that $p$ is {\it
    informative\/} if $p(a_0) = a_0$, $p(a_n) = a_n$ and there is a
  list $\{i_0, \ldots, i_r\}$ of indices such that
  \begin{enumerate}[(i)]
  \item $i_0 = 0$, $i_r = n$;
  \item $a_{i_k} = p(a_{i_k})$ for $0 \le k \le r$;
  \item for any $0 \le k < r$ the interval $(a_{i_k}, a_{i_{k+1}})$ is
    $p$-monotone.
  \end{enumerate}

  If $p$ is an informative partial isomorphism and $\dom{p} = \{a_0,
  \ldots, a_n \}$ as above then we set
  \[ \essen{p} = \big( \dom{p} \cup \ran{p} \big) \setminus \{a_0,
  a_n\}\] and refer to it as to {\it the set of essential points of
    $p$}.

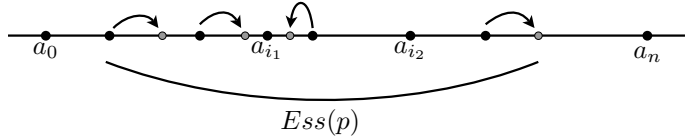
\bigskip

\input{informative.pst}

\bigskip
\end{definition}

\begin{definition}
  A pair $(p, q)$ of partial isomorphisms is called {\it piecewise
    elementary\/} if the following holds
  \begin{enumerate}[(i)]
  \item $p$ and $q$ are informative;
  \item $\min( \dom{p} ) = \min( \dom{q} )$,
  \item $\max( \dom{p} ) = \max( \dom{q} )$.
  \end{enumerate}
  If additionally $\fix{p} \cap \fix{q}$ has cardinality at most $2$
  (i.e., consists of the above minimum and maximum) then the pair $(p,q)$
  is called {\it elementary}.
\end{definition}

Let $(p, q)$ be a piecewise elementary pair, and $\fix{p} \cap \fix{q}
= \{a_0, \ldots, a_n\}$ with $a_i < a_j$ for $i< j$. Set $I_j = [a_j,
a_{j+1}]$, then $(p|_{I_j}, q|_{I_j})$ is elementary for any $0 \le j
< n$. Thus every piecewise elementary pair $(p,q)$ can be decomposed
into finitely many elementary pairs.

The following obvious lemma partially explains the importance of piecewise
elementary pairs.
\begin{lemma}\label{pelem-dense}
  For any open $V \subseteq G \times G$ there is a piecewise
  elementary pair $(p,q)$ such that $U(p) \times U(q) \subseteq V$.
\end{lemma}

\begin{definition}\label{def:elem}
  Let $(p,q)$ be an elementary pair. We say that a triple $(p', q',
  w)$ {\it liberates $p$ in $(p,q)$}, where $p'$ and $q'$ are partial
  isomorphisms that extend $p$ and $q$ respectively, and $w \in
  F(s,t)$ is a reduced word, if the following holds
  \begin{enumerate}[(i)]
  \item\label{def:elem:inf} $p'$ and $q'$ are informative;
  \item\label{def:elem:presmin} $\min( \dom{p'} ) = \min( \dom{p} )$,
    $\min( \dom{q'} ) = \min( \dom{q} )$;
  \item\label{def:elem:presmax} $\max( \dom{p'} ) = \max( \dom{p} )$,
    $\max( \dom{p'} ) = \max( \dom{p} )$;
  \item\label{def:elem:wform} the word $w$ starts from a non-zero
    power of $t$, $w = t^nv$ for $n \ne 0$;
  \item $w(p',q')(c)$ is defined for any $c \in \essen{p} \cup
    \essen{q}$ and
    \[ w(p',q')(\min( \essen{p} \cup \essen{q} )) > \max( \essen{p'}
    ),\]
  \item\label{def:elem:monot} there is an open interval $J$ such that $\rigint(J) =
    \max(\dom{q})$, $q'$ is monotone on $J$ and $w(p',q')(c) \in J$
    for any $c \in \essen{p} \cup \essen{q}$; moreover, if $n > 0$ in
    the item \eqref{def:elem:wform}, then $J$ is $q'$-increasing, and it
    is $q'$-decreasing otherwise. 
  \end{enumerate}
  Similarly, we say that a triple $(p',q',w)$ {\it liberates $q$ in
    $(p,q)$} if the above holds with roles of $p$ and $q$, $s$ and $t$
  interchanged.

  For a piecewise elementary pair $(p,q)$, we say that a triple
  $(p',q',w)$ liberates $p$ [liberates $q$] in $(p,q)$ if
  \begin{enumerate}[(i)]
  \item $\min( \dom{p'} ) = \min( \dom{p} )$, $\min( \dom{q'} ) =
    \min( \dom{q} )$;
  \item $\max( \dom{p'} ) = \max( \dom{p} )$, $\max( \dom{p'} ) =
    \max( \dom{p} )$;
  \item for any interval $I$, such that $(p|_I, q|_I)$ is elementary,
    the triple $(p'|_I, q'|_I, w)$ liberates $p|_I$ [liberates $q|_I$]
    in $(p|_I, q|_I)$.
  \end{enumerate}
\end{definition}

\begin{lemma}\label{liberating-elem-weak}
  For any elementary pair $(p,q)$ there is a triple $(p',q',w)$ that
  liberates $p$ [liberates $q$] in $(p,q)$.
\end{lemma}

\begin{proof}
  We show the existence of a triple that liberates $p$, the second
  clause then follows by symmetry.

  Extending $p$ and $q$ if necessary, we may assume that
  \begin{enumerate}[(i)]
  \item $\essen{p} \ne \es$, $\essen{q} \ne \es$;
  \item $I_1, \ldots, I_k$ are all the (open) intervals of
    monotonicity for $p$ and $J_1, \ldots, J_l$ are all the (open)
    intervals of monotonicity for $q$; we list them in the increasing
    order, i.e., $I_i < I_{i+1}$, $J_j < J_{j+1}$;
  \item\label{smp:nonemp} $I_1 \cap \dom{p} \ne \es$ and $J_1 \cap
    \dom{q} \ne \es$;
  \item $I_k \cap \dom{p} = \es$ and $J_l \cap \dom{q} = \es$;
  \item $\lefint(I_k) > \lefint(J_l)$.
  \end{enumerate}

  Let $\alpha = \min (\essen{p} \cup \essen{q})$. Then $\alpha \in I_1
  \cap J_1$ by \eqref{smp:nonemp} (and in particular $\alpha$ is not a
  fixed point of $p$ or $q$).  We first find an extension $p_1$ of $p$
  and $m_1 \in \Z$ such that $p_1^{m_1}(\alpha)$ is defined and is
  ``close enough'' to the right endpoint of $I_1$. ``Close enough''
  exactly means the following. Since by assumptions $\rigint(I_1)$ is
  not fixed by $q$ (because $(p,q)$ is elementary), there is some
  $j_1$ such that $\rigint(I_1) \in J_{j_1}$ and we want
  $p_1^{m_1}(\alpha) \in J_{j_1}$. Note that $p_1$ is informative and
  has the same intervals of monotonicity as $p$. At the second step we
  find extension $q_1$ of $q$ and $n_1 \in \Z$ such that
  $q_1^{n_1}p_1^{m_1}(\alpha)$ is defined and is ``close enough'' in
  the above sense to the right endpoint of $J_{j_1}$. We proceed in
  this way and stop as soon as the image of $\alpha$ reaches $J_l$,
  i.e, we obtain extensions $\bar{p}$, $\bar{q}$ of $p$ and $q$ and a
  word $u = s^{m_{N+1}}v$, where $v = t^{n_N}s^{m_N} \cdots
  t^{n_1}s^{m_1}$ such that $u(\bar{p}, \bar{q})(\alpha)$ is defined,
  lies in $J_l$ and $v(\bar{p}, \bar{q})(\alpha) \not \in J_l$. Note
  that since we added to the domain of $q$ only points of the orbit of
  $\alpha$ under $u$, this implies $\dom{\bar{q}} \cap J_l =
  \es$. Also by induction $\bar{p}$ and $\bar{q}$ are informative with
  the same decomposition into intervals of monotonicity as for $p$ and
  $q$.

  The following figure illustrates the construction (horizontal
  arrows indicate monotonicity of partial isomorphisms, bars stand for
  fixed points, the black dot is the minimal element $\alpha$, and
  gray dots are its images under $w$):

  
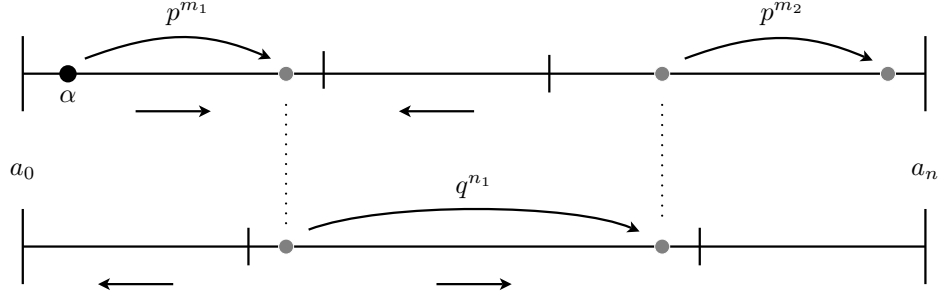
\bigskip

  \input{fig1.pst}

  \bigskip

  We now take extension $\bar{p}'$, $\bar{q}'$ of $\bar{p}$ and
  $\bar{q}$ such that
  \begin{enumerate}[(i)]
  \item $u(\bar{p}',\bar{q}')(c)$ is defined for every $c \in
    \essen{p} \cup \essen{q}$;
  \item $\bar{p}'$ and $\bar{q}'$ are informative with the same
    decomposition into intervals of monotonicity as for $\bar{p}$ and
    $\bar{q}$;
  \item minimum and maximum of the domains of $\bar{p}'$ and
    $\bar{q}'$ are equal to the minimum and maximum of the domains of
    $\bar{p}$ and $\bar{q}$;
  \item $\bar{q}'$ is increasing on $J_l$ (this is possible since $J_l \cap
    \dom{\bar{q}} = \es$).
  \end{enumerate}

  Set $p' = \bar{p}'$. Finally extending $\bar{q}'$ to $q'$ we can
  find $M \in \N$ such that
  \[ q'^M u(p',q')(\alpha) > \max( \essen{p'}).\] And so let $w =
  t^Mu$, then $(p',q',w)$ liberates $p$ in $(p,q)$.
\end{proof}

\begin{remark}\label{fg-word-sym}
  In the lemma above we started our construction from applying a power
  of $p$, but we likewise could start it from applying a power of $q$.
\end{remark}

\begin{remark}\label{rem:metcl}
  We view rationals as a dense linear ordering without endpoints. But
  note that if would have a usual metric on $\bbQ$ then the above
  construction gives us $p'$, $q'$, and $w$ such that
  $w(p',q')(\alpha)$ is as close in this metric to the endpoint
  $\max(\dom{p})$ as one wants. We will use this observation later.
\end{remark}

\begin{lemma}\label{liberating-elem-strong}
  For any elementary pair $(p,q)$ and any word $u$ there is a word
  $v$, and partial isomorphisms $p'$ and $q'$ such that the triple
  $(p',q',vu)$ liberates $p$ [liberates $q$] in $(p,q)$ and $|vu| =
  |v| + |u|$ (i.e., no cancellation between $v$ and $u$ happens).
\end{lemma}

\begin{proof}
  First we take extensions $p_1$ and $q_1$ of $p$ and $q$ respectively
  such that $u(p_1, q_1)(c)$ is defined for any $c \in \dom{p} \cup
  \dom{q}$, $(p_1,q_1)$ is elementary and 
  \[\min(\dom{p_1}) = \min(\dom{p}) = \min(\dom{q}) =
  \min(\dom{q_1}),\]
  \[\max(\dom{p_1}) = \max(\dom{p}) = \max(\dom{q}) =
  \max(\dom{q_1}).\]

  By Lemma \ref{liberating-elem-weak} one can find a word $v$ and
  extensions $p'$, $q'$ of $p_1$, $q_1$ such that $(p',q',v)$
  liberates $p_1$ in $(p_1,q_1)$. By Remark \ref{fg-word-sym} we may
  also assume that there is no cancellation in $vu$. Then $(p', q',
  vu)$ liberates $p$ in $(p,q)$.
\end{proof}

\begin{lemma}\label{liberating-preserve}
  Let $(p,q)$ be a piecewise elementary pair and assume a triple
  $(p',q',w)$ liberates $p$ [liberates $q$] in $(p,q)$. Let $u = t^nv$
  [$u = s^mv$] be a reduced word such that $uw$ is irreducible. Then
  there is a triple $(p'',q'',uw)$ that liberates $p$ [liberates $q$]
  in $(p,q)$. Moreover, one can take $p''$ to be an extension of $p'$
  and $q''$ to be an extension of $q'$.
\end{lemma}

\begin{proof}
  By the definition of liberation for piecewise elementary pairs it is
  enough to prove the statement for elementary triples only. So assume
  $(p,q)$ is elementary. Since $w$ liberates $p$ in $(p,q)$ it has to
  start with a non-zero power $l$ of $t$, i.e, $w=t^l*$. We prove the
  statement by induction on $|u|$. If $u$ is empty the statement is
  trivial. Now consider the inductive step. Either $u=*t^k$ and sign
  of $k$ matches the sign of $l$ (because $uw$ has to be reduced by
  assumptions) or $u = *s^k$ with $k \ne 0$. In the former case extend
  $q'$ to $q'_1$ in such a way that $(t^kw)(p',q'_1)(c)$ is defined
  for $c \in \essen{p} \cup \essen{q}$, then $(p',q_1',t^kw)$ will be
  a $p$ liberating tuple by the item \eqref{def:elem:monot} of the
  definition of liberation. In the second case we can find $p'_1$ such
  that $(p'_1,q',s^kw)$ liberates $q'$ in $(p',q')$ by taking $p'_1$
  such that $(s^kw)(p'_1,q')(c) > \max(\essen{p'} \cup \essen{q'})$ for any
  $c \in \essen{p'} \cup \essen{q'}$. This proves the step of
  induction and the lemma.
\end{proof}

\begin{lemma}\label{liberating-general}
  Let $(p,q)$ be a piecewise elementary pair and $u \in F(s,t)$. Then
  there is a triple $(p',q',w)$ that liberates $p$ [liberates $q$] in
  $(p,q)$ and such that $w=vu$ is reduced.
\end{lemma}

\begin{proof}
  We prove the statement by induction on the number of elementary
  components of $(p,q)$. Lemma \ref{liberating-elem-strong} covers the
  base of induction. Assume we have proved the lemma for $r$-many
  elementary components and inductively constructed a triple
  $(\bar{p}_r,\bar{q}_r,w_r)$ that liberates $p_r$ in $(p_r,q_r)$,
  where $p_r$ and $q_r$ are restrictions of $p$ and $q$ onto the
  first $r$ many elementary components. Consider the $r+1$ elementary
  component of $(p,q)$. By Lemma \ref{liberating-preserve} and by the
  base of induction there is a word $w_{r+1} = v_{r+1}w_r$ and
  extensions $\bar{p}_{r+1}$, $\bar{q}_{r+1}$ of $\bar{p}_r$ and
  $\bar{q}_r$ such that $(\bar{p}_{r+1},\bar{q}_{r+1},w_{r+1})$
  liberates $p_{r+1}$ in $(p_{r+1},q_{r+1})$. By remark
  \ref{fg-word-sym} we may assume that $v_{r+1}w_r$ is reduced.  This
  proves the step of induction and hence the lemma.
\end{proof}

\begin{lemma}\label{lem:density-words}
  For any pair $(p,q)$ of partial isomorphisms and any word $u \in
  F(s,t)$ there are extensions $p'$ and $q'$ of $p$ and $q$
  respectively and a reduced word $w=vu$ such that $w(p',q')(c) = c$
  for any $c \in \dom{p} \cup \dom{q}$.
\end{lemma}

\begin{proof}
  By Lemma \ref{pelem-dense} it is enough to prove the statement for a
  piecewise elementary pair $(p,q)$. By Lemma \ref{liberating-general}
  we can find extensions $\bar{p}$, $\bar{q}$ and a word $v$ such that
  $(\bar{p}, \bar{q}, vu)$ liberates $p$ in $(p,q)$. By the definition
  of liberation we can now extend $\bar{p}$ to $p'$ by declaring
  \[p'(c) = c,\ \textrm{for any $c \in vu(\bar{p}, \bar{q})(\essen{p}
    \cup \essen{q})$}.\] 
  Now set
  $q' = \bar{q}$ and $w = u^{-1}v^{-1}svu$. Then $w(p',q')(c) = c$ for
  any $c \in \dom{p} \cup \dom{q}$.
\end{proof}

\begin{lemma}\label{main-lemma}
  Fix a sequence $\{u_k\}$ of reduced words. For a generic $(f,g) \in
  G \times G$ there is a sequence of reduced words $w_k = v_ku_k$ such
  that $w_k(f,g) \ra \id$.
\end{lemma}

\begin{proof}
  Take an enumeration $\{c_i\} =
  \Q$ of the rationals. Let
  \[B_n^k = \{(f,g) \in G \times G : \exists w = vu_k\ \textrm{reduced
    and}\ w(f,g)(c_i) = c_i\ \textrm{for}\ 0\le i\le n \}.\] We claim
  that each $B_n^k$ is dense and open. Indeed, assume for a certain
  $n$ one has $(f,g) \in B_n^k$. This is witnessed by a word $w$. Set
  \[D = \cup_{i=0}^n \orb{\it w(f,g)}{c_i}\] and let $p = f|_D$, $q =
  g|_D$. Then $(f,g) \in U(p) \times U(q) \subseteq B_n^k$ and so
  $B_n^k$ is open. Density follows from Lemma \ref{lem:density-words}.

  Now by Baire theorem $\cap_{n, k} B_n^k$ is a dense
  $G_\delta$. The lemma follows.
\end{proof}

\begin{theorem}\label{thm:rationals}
  Each two-dimensional topological similarity class in $G$ is meager.
\end{theorem}

\begin{proof} Assume towards the contradiction that there is a pair
  $(f_1,g_1) \in G \times G$ that has a non-meager class of topological
  similarity. Then by Lemma \ref{main-lemma} there must be a sequence
  $w_n = v_nt^ns^n$ of reduced words such that $(f_1,g_1)$ converges to
  the identity along this sequence (we apply Lemma \ref{main-lemma}
  for the sequence $u_k = t^ks^k$).

  Take and fix $a \in \mathbb{Q}$. Set
  \[F_a = \{(f,g) \in G \times G : \textrm{$f(a) = a = g(a)$}\}.\]
  Let
  \[C_n = \{(x,y) \in G \times G : \exists m > n\ w_m(x,y)(a) \ne
  a\}.\] Then $C_n$ is open and dense in $(G \times G) \setminus
  F_a$. To see density take a basic open set
  $U(p) \times U(q) \subseteq (G \times G) \setminus F_a$ and assume
  $p(a) \ne a$ (the case when $p(a) = a$, but $q(a) \ne a$ is
  similar).
  For some $k > n$ $p^k(a)$ is not in the domain of $p$. Thus there
  are infinitely many values that $f \in U(p)$ may attain at
  $f^{k-1}(a)$, and so (by induction) there are infinitely many
  values that $w_k(f,g)(a)$ may attain for a pair
  $(f,g) \in U(p) \times U(q)$. Hence for some
  $(f,g)$ one has $w_k(f,g)(a) \ne a$. And so
  $C_n$ is dense in $G \times G \setminus F_a$. An application of
  Baire theorem shows that $\cap C_n$ is a dense $G_\delta$ and so for
  a generic $(f,g) \in (G \times G) \setminus F_a$ one has
  $w_n(f,g)(a) \not \ra a$. Since $\cup_a (G \times G)
  \setminus F_a = (G \times G) \setminus \{\id\times\id\}$ we get a
  contradiction with the assumption that $w_n(f_1,g_1) \to \id$ and that
  the class of topological similarity of $(f_1,g_1)$ is non-meager.
\end{proof}

\subsection*{Homeomorphisms of the unit interval.} We now turn to the
group of homeomorphisms of the unit interval. This is a Polish group
in the natural topology, given by the basic open sets:
\[U(f; a_1, \ldots, a_n; \eps) = \{g \in \homeo{[0,1]} : |g(a_i) -
f(a_i)| < \eps\}.\] We may write this neighborhood as $U(p; \eps)$,
where $p = f|_{\{a_0, \ldots, a_n\}}$ is a partial isomorphism. Since
$\Q$ is dense in $[0,1]$, we may assume that $p$ is a partial isomorphism
of the rationals: this will give us a base of open sets.

This group $\homeo{[0,1]}$ has a normal subgroup of index $2$, namely
the subgroup $\homeop{[0,1]}$ of order preserving homeomorphisms. Let
$H = \homeop{[0,1]}$, then $\aut{\Q} = G$ naturally embeds into $H$
(this embedding is a continuous injective homomorphism, its inverse,
though, is not continuous), and the image of
$G$ under this embedding is dense in $H$.

\begin{theorem}
  Every two-dimensional class of topological similarity in
  $H$ is meager.
\end{theorem}

\begin{proof}
  We imitate the proof of Theorem \ref{thm:rationals}. Let $\{x_m\}$ be an
  enumeration of the rationals $\Q \cap [0,1]$, then $\{x_m\}$ is dense in
  $[0,1]$. Set
  \[A_{m,n} = \{f \in H : |f(x_m) - x_m| > 1/n\ \textrm{and}\
  |f^{-1}(x_m) - x_m| > 1/n\},\]
  \[B_{m,n} = \{(f,g) \in H \times H : f \in A_{m,n}\ \textrm{or}\ g
  \in A_{m,n}\}.\] Note that $B_{m,n}$ is open for every $m$ and
  $n$. Then $\cup_{m,n} B_{m,n} = H \times H \setminus \{(\id,\id)\}$
  and so it is enough to prove that each two-dimensional class of
  topological similarity is meager in each of $B_{m,n}$. 

  Let $u_k$ be a sequence of words
  such that for every piecewise elementary pair $(p,q)$ (here $p$ and
  $q$ are partial isomorphisms of the rationals, like before) there
  are infinitely many $k$ such that for some $p'_k$, $q'_k$,
  $(p'_k,q'_k,u_k)$ liberates $p$ in $(p,q)$. Then by Lemma
  \ref{main-lemma} for a generic pair $(f,g) \in G \times G$ there is
  a sequence of reduced words $w_k = v_ku_k$ such that $w_k(f,g) \ra
  \id$. This implies that for a generic pair $(f,g) \in H \times H$
  there is a sequence $w_k$ as above (because the topology in $H$ is
  coarser than in $G$). If there is a non-meager
  two-dimensional class of topological similarity then there is a
  sequence of reduced words $\{w_k\} = \{v_ku_k\}$ (for some
  $\{v_k\}$) such that the set of pairs $(f_1,g_1) \in H \times H$
  that converges to the identity along $w_k$ is non-meager.

  Fix now $m,n$ and a sequence of reduced words $w_k = v_ku_k$. Set
  \[C_k = \{(f,g) \in H \times H : \exists K > k\ |w_K(f,g)(x_m) -
  x_m| > 1/2n\}.\] Each $C_k$ is open, we claim that it is also dense
  in $B_{m,n}$. Let $V \subseteq B_{m,n}$ be an open set. Without loss
  of generality we may assume that $V = U(p;\eps_1) \times
  U(q;\eps_2)$, where $p$ and $q$ are partial isomorphisms of the
  rationals. Let 
  \[\delta = \min \{|x_m - c| : c \in \fix{p} \cap \fix{q}\} > 1/n.\] 
  Then there is $K > k$ and $p'$,
  $q'$ such that $(p',q',u_K)$ liberates $p$ in $(p,q)$. 
  Now repeat the proof of Lemma \ref{liberating-preserve} and use
  Remark \ref{rem:metcl} to get $p''$, $q''$ that extend $p'$ and $q'$
  and such that $|w_K(p'',q'')(x_m) - x_m| \ge 1/2\delta$. Hence each
  $C_k$ is dense in $B_{m,n}$. Now by Baire theorem the intersection
  $\cap_k C_k$ is a dense $G_\delta$ in $B_{m,n}$ and thus for any
  specific sequence $w_k$ the set of elements $(f_1,g_1) \in H \times
  H$ that converges to the identity along this sequence is meager in
  $B_{m,n}$. Finally we showed that each two-dimensional topological
  similarity class is meager in $B_{m,n}$ for any $m$, $n$ and so is
  in $H \times H$.
\end{proof}

\section{Extensions of Partial Isometries}
\label{sec:extens-part-isom}

In this section we prove several results, that will be used later,
when dealing with the randomly ordered Urysohn space. But we believe
that some of the theorems below are of independent interest for
understanding the group of isometries of the Urysohn space.

We remind that the Urysohn space $\bbU$ is a complete separable metric
space, that is uniquely characterized by the following properties:
\begin{itemize}
\item Every finite metric space can be isometrically embedded into $\bbU$;
\item $\bbU$ is ultrahomogeneous, that is each partial isometry
  between finite subsets of $\bbU$ extends to a full isometry of $\bbU$.
\end{itemize}
There is a rational counterpart $\bbQU$ of the Urysohn space. It is
called rational Urysohn space. This is
a countable metric space with rational distances, characterized by
similar properties:
\begin{itemize}
\item Every finite metric space with rational distances can be
  isometrically embedded into $\bbQU$;
\item $\bbQU$ is ultrahomogeneous.
\end{itemize}

Groups of isometries $\iso{\bbU}$ and $\iso{\bbQU}$ of these spaces are
Polish groups when endowed with the topology of pointwise convergence
(for this $\bbQU$ is viewed as a discrete topological space). 

\begin{definition}
  Let $(\bA,d)$ be a finite metric space with at least two
  elements. The {\it density\/} of $\bA$ is denoted by $\den(\bA)$ and
  is the minimal distance between two distinct points in $\bA$:
  \[\den(\bA) = \min\{d(x,y) : x, y \in \bA, x \ne y\}.\]
\end{definition}

\begin{definition}
  \label{sec:extens-part-isom-ordmetric}
  An ordered metric space is a triple $(\bA, d, <)$, where $d$ is a
  metric on $\bA$ and $<$ is a linear ordering on $\bA$. 
\end{definition}

\begin{definition}
  A {\it partial isometry} or {\it partial isomorphism} of a metric
  space $\bC$ is an isometry $p : \bA \to \bB$ between {\it finite\/}
  subspaces $\bA$, $\bB \subseteq \bC$. A partial isomorphism of and
  ordered metric space is a partial isometry of the metric space that
  also preserves the ordering on its domain.
\end{definition}

\begin{definition}
  \label{sec:extens-part-isom-perpoint}
  Let $p$ be a partial isometry of a metric space. Then we let
  $\dom{p}$ to denote the domain of $p$ and $\ran{p}$ to denote its
  range. A point $x \in \dom{p}$ is called {\it periodic\/} if there
  is a natural number $n > 0$ such that 
  \[x, p(x), \ldots, p^n(x) \in \dom{p} \quad \textrm{and} \quad p^n(x) = x.\]
  The set of periodic points is denoted by
  $\per{p}$. A point $x \in \dom{p}$ is called {\it fixed\/} if $p(x) =
  x$ and the set of fixed points is denoted by $\fix{p}$.
\end{definition}

In this section we deal mostly with the classical Urysohn space, but
some of the results will be later applied to the randomly ordered
rational Urysohn space. The following proposition will let us do that.

\begin{proposition}
  \label{sec:isom-to-ord}
  Let $\bA$ be a finite ordered metric space, let $p$ be a partial
  isomorphism of $\bA$. Let $\bB$ be a finite metric space (with no
  ordering) and let $q$ be a partial isometry of $\bB$ with $\per{q} =
  \fix{q}$. Suppose that $\bA \subseteq \bB$ as metric spaces and $q$
  extends $p$. If
  \[\forall x \in \dom{q}\ q(x) \in \bA \iff x \in \dom{p}\]
  then there is a linear ordering on $\bB$ that
  extends an ordering on $\bA$ and such that $q$ becomes a partial
  isomorphism of an ordered metric space $\bB$.
\end{proposition}

\begin{proof}
  We prove the statement by induction on $|\bB \setminus \bA|$. If $\bA
  = \bB$ the statement is obvious. For inductive step we consider two
  cases.

  {\bf Case 1.} There is some $x \in \bA$ such that $x \in \dom{q}$
  but $x \not \in \dom{p}$. Then by assumption $q(x) \in \bB \setminus
  \bA$. Extend now a linear ordering on $\bA$ to a partial ordering on
  $\bA \cup \{q(x)\}$ by declaring for $y \in \bA$
  \[q(x) < y \iff \exists z \in \dom{p}\ (p(z) \le y) \& (x < z),\]
  \[y < q(x) \iff \exists z \in \dom{p}\ (y \le p(z)) \& (z < x).\]
  It is straightforward to check that this relation is indeed a
  partial ordering on $\bA \cup \{q(x)\}$. Extend this partial
  ordering to a linear ordering on $\bA \cup \{q(x)\}$ in any
  way. Then $q$ is a partial isomorphism of $\bA \cup \{q(x)\}$ and we
  apply the induction.

  {\bf Case 2.} Assume the opposite to the first case happens. Then
  $q|_{\bA} = p$. Take any $x \in \dom{q} \setminus \bA$ (if there is
  no such $x$ then $\dom{p} = \dom{q}$ and the statement is
  obvious). Assume first that $x$ is not a fixed point of $q$. Then
  define a linear ordering on $\bA \cup \{x, q(x)\}$ by declaring
  \[ \forall y \in \bA\ (y < x) \& (y < q(x)) \& (x < q(x)).\] Then
  $q$ is a partial isomorphism of $\bA \cup \{x, q(x)\}$ and we can
  apply the induction hypothesis. If $x$ was a fixed point then we
  declare
  \[ \forall y \in \bA\ (y < x),\]
  and, again, induction does the rest.
\end{proof}

\begin{definition}
  Let $\bA = (\bA, d_{\bA})$, $\bB = (\bB, d_{\bB})$, and $\bC =
  (\bC, d_{\bC})$ be finite metric spaces and $i : \bA \to \bB$, $j
  : \bA \to \bC$ be isometries. We define {\it free amalgamation} $\bD =
  \bB *_{\bA}{}\bC$ of metric spaces as follows: substituting $\bB$
  and $\bC$ by their isomorphic copies we may assume that $\bB
  \cap \bC = \bA$, set $\bD = \bB \cup \bC$ and define the metric
  $d_{\bD}$ by:
  \begin{equation*}
    d_{\bD}(x,y) = 
    \begin{cases} 
      d_{\bB}(x,y) & \text{if $x$, $y \in \bB$,} \\
      d_{\bC}(x,y) & \text{if $x$, $y \in \bC$,} \\
      \min\limits_{z \in \bA}\{d_{\bB}(x,z) + d_{\bC}(z,y)\} &
      \text{otherwise.}
    \end{cases}
  \end{equation*}
  Note that the first and the second clauses agree for $x$, $y \in
  \bA$.

  If $\bA$ is empty then we set $R = \diam(\bB) + \diam(\bC)$, 
  $\bD = \bB \sqcup \bC$ and
  \begin{equation*}
    d_{\bD}(x,y) = 
    \begin{cases} 
      d_{\bB}(x,y) & \text{if $x$, $y \in \bB$,} \\
      d_{\bC}(x,y) & \text{if $x$, $y \in \bC$,} \\
      R & \text{otherwise.}
    \end{cases}
  \end{equation*}
\end{definition}

The core of our arguments will be the following seminal result due to
\ssolecki established in 2005, see \cite{S}. The second item is
slightly modified compared to the original statement, but the
modification follows from the proof in \cite{S} without any additional
work. 

\begin{theorem}[\solecki]\label{thm:solecki}
  Let a finite metric space $\bA$ and a partial isometry $p$ of $\bA$
  be given. There exist a finite metric space $\bB$ with $\bA
  \subseteq \bB$ as metric spaces, an isometry $\bar{p}$ of $\bB$
  extending $p$, and a natural number $M$ such that
  \begin{enumerate}[(i)]
  \item $\bar{p}^{2M} = \id_B$;
  \item if $a \in \bA$ is aperiodic then $\bar{p}^j(a) \ne a$ for $0 <
    j < 2M$, and moreover for any $0 < j < 2M$ $\bar{p}^j(a) \in \bA$
    iff $\bar{p}^{j-1}(a) \in \dom{p}$;
  \item $\bA \cup \bar{p}^M(\bA)$ is the free amalgam of $\bA$ and
    $\bar{p}^M(\bA)$ over $(\per{p}, id_{\per{p}}, \bar{p}^M |_{\per{p}})$.
  \end{enumerate}
  Moreover, the distances in $\bB$ may be taken from the additive
  semigroup generated by the distances in $\bA$.
\end{theorem}


\begin{definition}
  \label{sec:isom-rand-order-disjstruct}
  Let $\bA$, $\bB$, $\bC$ be metric spaces and let $\bC$ be embedded
  into $\bA$ and $\bB$. We say that $\bB$ {\it extends $\bA$ over\/}
  $\bC$ 
  if there exist and embedding $i : \bA \to \bB$ such that the
  following diagram commutes:
  \[\xymatrix{\bA \ar@{-->}[dr]^{i} &  \\
              \bC \ar@{^{(}->}[u] \ar@{^{(}->}[r] &  \bB}\]

  We say that $\bA$ and $\bB$ are {\it disjoint over\/}
  $\bC$ if neither $\bB$ extends $\bA$ over $\bC$ nor $\bA$ extends
  $\bB$ over $\bC$.
\end{definition}

\begin{lemma}
  \label{sec:isom-rand-order-diffext}
  Let $\bA$ be a finite metric space, let $p$ be a partial
  isometry of $\bA$, let $x \in \dom{p}$ be such that $x \not \in
  \fix{p}$ (i.e., $p(x) \ne x$). Then there are metric spaces
  $\bA_1$ and $\bA_2$ that both extend $\bA$: $\bA \subset \bA_1$ and
  $\bA \subset \bA_2$, and partial isometries $p_1$ of $\bA_1$ and
  $p_2$ of $\bA_2$ that both extend $p$ and such that $\orb{p_1}{x}$
  and $\orb{p_2}{x}$ are disjoint over $\orb{p}{x}$.

  Moreover, one can assume that
  \[\per{p_1} = \per{p} = \per{p_2},\]
  \[\forall x \in \dom{p_1}\ p_1(x) \in \bA \iff x \in \dom{p},\]
  \[\forall x \in \dom{p_2}\ p_2(x) \in \bA \iff x \in \dom{p}.\]
\end{lemma}

\begin{proof}
  Apply Theorem \ref{thm:solecki} to get a full
  isometry $\bar{p}$ of a finite metric space $\bB$ that extends
  $p$ and a natural number $M$. Set
  \[ \bar{\bA} = \bA \cup \bar{p}(\bA) \cup \ldots \cup
  \bar{p}^{2M-1}(\bA) \cup \{y\},\] where $y$ is a new point, i.e., a
  point not in $\bB$. Let $\delta = \den(\bA)$ denote the density of $\bA$ and fix
  an $\eps > 0$ such that $\eps \le 2\delta$. We turn $\bar{\bA}$ into a
  metric space by defining the distance between $a,b \in \bar{\bA}$, $a
  \ne b$ as follows.
  \begin{equation*}
  d_{\bar{\bA}}(a,b) = 
  \begin{cases}
  d_{\bB}(a,b) & \textit{if $a \ne y$ and $b \ne y$},\\
  d_{\bB}(a,x) & \textit{if $a \ne x$ and $b = y$},\\
  d_{\bB}(x,b) & \textit{if $a = y$ and $b \ne x$},\\
  \eps & \textit{if $\{a,b\} = \{x,y\}$}.
  \end{cases}
\end{equation*}
  We claim that $(\bar{\bA}, d_{\bar{\bA}})$ is a metric space. We
  have to check the triangle inequality (other conditions are
  obviously fulfilled). For this note that both $\bar{\bA} \setminus
  \{y\}$ and $\bar{\bA} \setminus \{x\}$ are isometrically embeddable
  into $\bB$, where the triangle inequality is known to be
  satisfied. So one needs to prove two claims.

  \textbf{Claim 1.} For any $z \in \bar{\bA}$
  \[d_{\bar{\bA}}(x,y) \le d_{\bar{\bA}}(x,z) + d_{\bar{\bA}}(z,y).\]

  If $z \in \{x,y\}$ then the statement is obvious. If $z \not \in
  \{x,y\}$ then $d_{\bar{\bA}}(x,z) + d_{\bar{\bA}}(z,y) \ge 2\delta$
  and $d_{\bar{\bA}(x,y)} = \eps \le 2\delta$ and Claim 1 follows.

  \textbf{Claim 2.} For any $z \in \bar{\bA}$
  \[d_{\bar{\bA}(x,z)} \le d_{\bar{\bA}}(x,y) + d_{\bar{\bA}}(y,z),\]
  \[d_{\bar{\bA}(z,y)} \le d_{\bar{\bA}}(z,x) + d_{\bar{\bA}}(x,y).\]

  Note that for $z \not \in \{x,y\}$ one has $d_{\bar{\bA}}(y,z) = d_{\bar{\bA}}(x,z)$. From this both
  inequalities follow immediately. 

  So $\bar{\bA}$ is a metric space. We denote it by $\bar{\bA}(\eps)$
  to signify the dependence on epsilon. Define a partial isometry
  $\hat{p}$ on $\bar{\bA}(\eps)$ by 
  \[ \hat{p}(x) = \bar{p}(x),\] whenever $x \in \bar{\bA}$ and
  $\bar{p}(x) \in \bar{\bA}$; and $\hat{p}(\hat{p}^{2M-1}(x)) = y$. It
  is straightforward to check that $\hat{p}$ is indeed a partial
  isometry. 
  Now the construction of two extensions that are
  disjoint over $\orb{p}{x}$ is easy. Take, for example, two different
  $\eps_1 \le 2\delta$, $\eps_2 \le 2\delta$, $\eps_1 \ne \eps_2$ such
  that 
  \[\eps_i \not \in \{d_\bB(x_1, x_2) : x_1, x_2 \in \bB\},\] let
  $(\bA_i, p_i) = (\bar{\bA}(\eps_i), \hat{p})$. Then $\orb{p_1}{x}$
  and $\orb{p_2}{x}$ are disjoint over $\orb{p}{x}$.
\end{proof}

The main power of Theorem \ref{thm:solecki} is the explicit construction of an
extension of a partial isometry to a full isometry of a finite metric
space. Moreover, this extension is as independent as possible. For our
purposes we need only an extension to a partial isomorphism, but we want
to keep the independence. 
Let us state explicitly a corollary of the
theorem that keeps all we need.

\begin{corollary}\label{cor:slawek}
  For any finite metric space $\bA$ and a partial isometry $p$ there
  is finite metric space $\bC$, a partial isometry $p_1$ of $\bC$,
  which is an extension of $p$, and a natural number $M$ such that
  \begin{enumerate}[(i)]
  \item $\per{p} = \per{p_1}$;
  \item $\bA \cup p_1^M(\bA)$ is the amalgam of $\bA$ and
    $p_1^M(\bA)$ over $(\per{p}, id_{\per{p}}, p_1^M |_{\per{p}})$.
  \item for any $x \in \dom{p_1}$
  \[p_1(x) \in \bA \iff x \in \dom{p}.\]
  \end{enumerate}
  Moreover, the distances in $\bC$ are taken from the
  additive semigroup generated by the distances in $\bA$, and hence
  the density is preserved $\den(\bC) = \den(\bA)$. 
\end{corollary}

\begin{proof}
  Apply Theorem \ref{thm:solecki} to $\bA$ and $p$ to get a metric
  space $\bB$, a full isometry $\bar{p}$ of $\bB$ and a natural
  number $M$. Now set 
  \[\bC = \bA \cup \bar{p}(\bA) \cup \ldots \bar{p}^M (\bA),\]
  and $p_1 = \bar{p}|_{\bA \cup \bar{p}(\bA) \cup \ldots \cup
    \bar{p}^{M-1}(\bA)}$.
  It is trivial to
  check that such a $\bC$ and $p_1$ satisfy the conditions. 
\end{proof}

\begin{definition}
Let $(M,d)$ be a metric space, let $x,y \in M$. We say that the
distance $d(x,y)$ {\it passes through a point\/} $z \in M$ if 
\[d(x,y) = d(x,z) + d(z,y).\]
\end{definition}

We are going to apply Corollary \ref{cor:slawek} to partial
isometries that also preserve an ordering. That is why we impose
additional assumption: all periodic points are fixed points, i.e.,
$\per{p} = \fix{p}$.

\begin{theorem}\label{thm:isomext}
  Let $\bA$ be a finite metric space. Let $p$ and $q$ be two partial
  isometries of $\bA$ such that $\per{p} = \fix{p}$ and $\per{q} =
  \fix{q}$. Suppose $\fix{p} \cap \fix{q} \ne \es$. Then there are
  finite metric space $\bB$,
  extensions $\bar{p}$, $\bar{q}$ of $p$ and $q$ respectively (these
  extensions are partial isometries of $\bB$) and an
  element $w = t^Kv \in F(s,t)$, $K \ne 0$ such that
  \begin{enumerate}[(i)]
  \item\label{thm:prop:fix} $\per{\bar{p}} = \per{p}\ (= \fix{p})$, $\per{\bar{q}}
    = \per{q}\ (= \fix{q})$; 
  \item\label{thm:prop:amal} $\dom{\bar{p}} \cup w(\bar{p},
    \bar{q})(\bA)$ is the free amalgam of $\dom{\bar{p}}$ and
    $w(\bar{p},\bar{q})(\bA)$ over $F(p) \cap F(q)$.
  \end{enumerate}
  Moreover, the distances in $\bB$ are taken from the
  additive semigroup generated by the distances in $\bA$, and hence
  $\den(\dom{\bar{p}} \cup \bA) = \den(\bA)$, $\den(\dom{\bar{q}} \cup
  \bA) = \den(\bA)$.
\end{theorem}

\begin{proof}
  Let 
  \[ N = \left\lceil \frac{2\diam(\bA)}{\den(\bA)} \right\rceil.\] Define
  inductively the sequence of elements $w_k \in F(s,t)$, extensions
  $\bar{p}_k$, $\bar{q}_k$ and metric spaces $\bA_k$ as follows:

  {\bf Step 0:} Let $\bar{p}_0 = p$, $\bar{q}_0 = q$, $w_0 = $ empty word,
  $\bA_0 = \bA$;

  {\bf Step k:} If $k$ is odd then apply Corollary \ref{cor:slawek} to
  $\bar{p}_{k-1}$ and $\bA_{k-1}$ to get $\bar{p}_k$ and $M_k$; set
  $\bar{q}_k = \bar{q}_{k-1}$, $w_k = s^{M_k}w_{k-1}$, $\bA_k =
  \bA_{k-1} \cup \dom{\bar{p}_{k}} \cup \ran{\bar{p}_{k}}$.

  If $k$ is even do the same thing with roles of $p$ and $q$
  interchanged.

  We claim that $\bar{p} = \bar{p}_{2N+2}$, $\bar{q} =
  \bar{q}_{2N+2}$, $\bB = \bA_{2N+2}$, and $w = w_{2N+2}$ fulfill the
  requirements of the statement. Let $d$
  denote the metric on $\bB$. It is obvious that $\fix{\bar{p}}
  = \fix{p}$ and $\fix{\bar{q}} = \fix{q}$ (this is given by Corollary
  \ref{cor:slawek} at each stage). 
  The moreover part is also obvious, since it is fulfilled
  at every step of the construction.
  It remains to show that for any $x \in w(\bar{p}, \bar{q})(\bA)$ and
  any $y \in \dom{\bar{p}}$ one has
  \[d(x,y) = \min\{d(x,z) + d(z,y) : z \in \fix{p} \cap
  \fix{q}\}. \eqno{(1)}\] Note that by the last step of the
  construction for any $x \in w(\bar{p}, \bar{q})(\bA)$ and $y \in
  \dom{\bar{p}}$ we have
  \[d(x,y) = \min\{d(x,z) + d(z,y) : z \in \fix{q}\}.\]

  We first prove several claims.

  \textbf{Claim 1.} It is enough to show that $(1)$ holds for all $x \in
  w(\bar{p}, \bar{q})(\bA)$ and $y \in \fix{q}$.

  \textit{Proof of Claim 1.} Assume $(1)$ holds for all $x \in
  w(\bar{p}, \bar{q})(\bA)$ and $y \in \fix{q}$. Let $y'
  \in \dom{\bar{p}}$, then for some $c \in \fix{q}$
  \[d(x, y') = d(x,c) + d(c,y') = \min\{d(x,e) +
  d(e,y') : e \in \fix{q}\}. \eqno{(2)}\]
  By assumptions of the claim we get 
  \[d(x, y') = d(x,z) + d(z,c) + d(c, y') \ge d(x,z) +
  d(z,y'),\]
  for some $z \in \fix{p} \cap \fix{q}$; and so, by $(2)$,
  \[d(x,y') = d(x,z) + d(z,y').\]
  This proves the claim.

  Let $w_i(c)$ denote
  $w_i(\bar{p}_i, \bar{q}_i)(c)$. 

    \textbf{Claim 2.} Let $x \in \fix{p} \cup \fix{q}$, $c \in \bA$
    and suppose that for some $z \in \fix{p} \cap \fix{q}$ and for
    some $i$ the distance between $w_i(c)$ and $x$ passes through
    $z$. Then for any $j \ge i$ the distance between $w_j(c)$ and $x$
    passes through the same point $z$.

  \textit{Proof of Claim 2.} This follows by induction. Here is an
  inductive step. 
  Assume for
  definiteness that $j+1$ is odd (the case, when $j+1$ is even, is
  similar). The distance between $x$ and $w_{j+1}(c)$ passes through
  a point $z' \in \fix{p}$ ($z' \in \fix{q}$ if $j + 1$ is
  even). Then
  \[d(w_{j}(c),x) = d(w_{j}(c),z) + d(z,x) \le
  d(w_{j}(c),z') + d(z',x),\]
  \[d(w_{j+1}(c),x) = d(w_{j+1}(c),z') + d(z',x),\]
  but $d(w_{j+1}(c),z') = d(w_{j}(c),z')$ (this is because
  $w_{j+1} = s^mw_{j}$ and $z'$ is fixed by $p$). Hence
  \[d(w_{j}(c),x) \le d(w_{j+1}(c),x),\]
  but also 
  \[d(w_{j+1}(c),x) \le d(w_{j+1}(c),z) + d(z,x)
  = d(w_{j}(c),z) + d(z,x) = d(w_{j}(c),x),\] and so
  $d(w_{j+1}(c),x) = d(w_{j}(c),x)$. This proves the claim.

  \textbf{Claim 3.} Let $x \in \fix{p} \bigtriangleup \fix{q}$ (here
  $\bigtriangleup$ is a symmetric difference of sets), $c \in
  \bA$. Suppose that the distance between $w_{i}(c)$ and $x$ doesn't
  pass through a point in $\fix{p} \cap \fix{q}$. Then $d(w_i(c),x)
  \ge \lfloor i/2 \rfloor \den(\bA)$.

  \textit{Proof of Claim 3.}
  Suppose first that $x \in \fix{p} \setminus
  \fix{q}$. 
  We prove the statement by induction on $i$. The base of induction is
  trivial, so we show the inductive step: the statement is true for
  $i$ and we need to show it for $i+1$. If $i$ is even then, since
  $\lfloor i/2 \rfloor = \lfloor (i+1)/2 \rfloor$ and because
  $d(w_{i+1}(c),x) = d(w_{i}(c),x)$ (this is since $i$ is even
  and $x \in \fix{p}$) the statement follows immediately. So, assume
  $i$ is odd. Then the distance between $w_{i+1}(c)$ and $x$ passes
  through a point $z \in \fix{q}$. Now two things can happen. Suppose
  first for some $j \le i$ the distance between $w_j(c)$ and $z$
  passes through a point $z' \in \fix{p} \cap \fix{q}$. Then by Claim
  2, the distance between $z$ and $w_{i+1}(c)$ must pass through
  $z'$. Now
  \begin{multline*}
  d(w_{i+1}(c),x) = d(w_{i+1}(c),z) + d(z,x) = \\
  d(w_{i+1}(c),z') + d(z',z) + d(z,x) \ge d(w_{i+1}(c),z')
  + d(z',x).
  \end{multline*}

  And so the distance between $w_{i+1}(c)$ and $x$ passes through a
  point $z' \in \fix{p} \cap \fix{q}$. Contradiction with the assumptions
  of the claim. So, for no $j \le i$ the distance between $w_j(c)$ and $x$
  passes through a point in $\fix{p} \cap \fix{q}$. Then, applying
  induction to $w_i(c)$ and $z$, we get $d(w_i(c),z) \ge \lfloor
  i/2 \rfloor \den(\bA)$. But since $d(w_{i+1}(c),z) = d(w_i(c),z)$
  and since $d(x,z) \ge \den(\bA)$ we get
  \[d(w_{i+1}(c),x) \ge \lfloor i/2 \rfloor \den(\bA) + \den(\bA) \ge \lfloor
  (i+1)/2 \rfloor \den(\bA).\]
  In the case, when $x \in \fix{q} \setminus \fix{q}$, the distance
  increases by $\den(\bA)$ at even stages of the construction, the rest of
  the argument for this case is similar. The claim is proved.

  Fix now $c \in \bA$ and $y \in \fix{q}$. 
  It remains to show that
  \[d(w_{2N+2}(c),y) = \min\{d(w_{2N+2}(c),z) + d(z,y) : z \in \fix{p} \cap
  \fix{q}\}.\]  


  We now have two cases (we'll show,
  though, that Case $2$ is impossible).

  \textbf{Case 1.} For some $i \le 2N+2$ the distance between
  $w_{i}(c)$ and $y$ passes through a point $z \in \fix{p} \cap
  \fix{q}$. Then
  \[d(y,w_{i}(c)) = \min\{d(y,z) + d(z,w_{i}(c)): z \in
  \fix{p} \cap \fix{q}\}.\] 

  Applying Claim 2 for $j = 2N+2$, we get \[d(y, w_{2N+2}(c))
  = \min\{d(y,z) +
  d(z,w_{2N+2}(c)) : z \in \fix{p}\cap \fix{q}\}.\] And the theorem
  is proved for this case.

  \textbf{Case 2.} For no $i \le 2N+2$ the distance between
  $w_i(c)$ and $y$ passes through a point in $\fix{p} \cap
  \fix{q}$. Then by Claim 3 
  \[d(w_{2N+2}(c),y) \ge (N+1)\den(\bA) > 2diam(\bA).\]
 
  Let, on the other hand, $z \in \fix{p} \cap \fix{q}$ be any common
  fixed point. Then $d(w_{2N+2}(c),y) \le d(w_{2N+2}(c),z) +
  d(z,y) = d(c,z) + d(z,y) \le 2diam(\bA)$. Contradiction. So
  this case never happens.
%
%
%
\end{proof}

\begin{remark}\label{sec:main-remark}
  Note that the same result is also true for ordered metric
  spaces. For this one just has to apply Proposition
  \ref{sec:isom-to-ord} at each step of the construction of $\bar{p}$
  and $\bar{q}$.
\end{remark}

Before we apply this result to the classes of topological similarity
let us mention another application. For a subset $\bA \subseteq \bbU$
($\bA \subseteq \bbQU)$ let $\isom{\bA}{\bbU}$ ($\isom{\bA}{\bbQU}$,
respectively) denote the subgroup of isometries that pointwise fix
$\bA$.  Recall a theorem of Julien Melleray from \cite{M}.

\begin{theorem}[Melleray]
  \label{sec:isom-rand-order-gm}
  Let $\bbU$ be the Urysohn space, let $\bA, \bB \subset
  \bbU$ be two finite subsets. Then
  \[\isom{\bA \cap \bB}{\bbU} = \overline{ \langle \isom{\bA}{\bbU},
    \isom{\bB}{\bbU} \rangle}.\]
\end{theorem}

Let us give an equivalent reformulation of the above result.

\begin{theorem}[Melleray]
  \label{sec:isom-rand-order-gm2}
   Let $\bbU$ be the Urysohn space, let $\bA, \bB \subset
  \bbU$ be two finite subsets. Then for any $\eps > 0$, for any $p \in
  \isom{\bA \cap \bB}{\bbU}$, for any $\bC \subseteq \bbU$ there is $q \in \langle \isom{\bA}{\bbU},
    \isom{\bB}{\bbU} \rangle$ such that 
  \[\forall x \in \bC\ d(p(x),q(x)) < \eps.\]
\end{theorem}

We show that one can actually eliminate the epsilon in the above
reformulation.

\begin{theorem}
  \label{sec:isom-rand-order-noeps}
   Let $\bbU$ be the Urysohn space, let $\bA, \bB \subset
  \bbU$ be two finite subsets. Then for any $p \in
  \isom{\bA \cap \bB}{\bbU}$ for any $\bC \subseteq \bbU$ there is $q \in \langle \isom{\bA}{\bbU},
    \isom{\bB}{\bbU} \rangle$ such that 
  \[\forall x \in \bC\ p(x) = q(x).\]
\end{theorem}

\begin{proof}
  Without loss of generality we may assume that $\bA \subseteq \bC$ and
  $\bB \subseteq \bC$. Let $\bD = \bC \cup p(\bC)$ then $p|_{\bC}$ is a
  partial isometry of $\bD$. Define two partial isometries $p_1$ and
  $p_2$ of $\bD$ by
  \[\forall x \in \bA\ p_1(x) = x,\]
  \[\forall x \in \bB\ p_2(x) = x.\]
  Apply now Theorem \ref{thm:isomext} to $p_1$, $p_2$ and $\bD$ to get
  a metric space $\bD'$ and extension $q_1$ of $p_1$ and $q_2$ of
  $p_2$, and a word $w \in F_2$. Extend now $q_1$ by
  \[\forall x \in \bC\ q_1(w(q_1,q_2)(x)) = w(q_1,q_2)(p(x)).\]
  Such a $q_1$ is then a partial isometry of $\bD'$. Now extend $q_1$
  and $q_2$ to full isometries (we still denote them by the same
  symbols) and set
  \[q = w^{-1}(q_1,q_2)q_1 w(q_1,q_2).\]
  Then for any $x \in \bC$, $p(x) = q(x)$, and $q_1 \in
  \isom{\bA}{\bbU}$, $q_2 \in \isom{\bB}{\bbU}$.
\end{proof}

Note that if we start from metric spaces with rational
distances, then the space $\bD'$, constructed in the proof, would also
have rational distances. And we arrive at the

\begin{corollary}
  Let $\bbQU$ be the rational Urysohn space, let $\bA, \bB \subset
  \bbQU$ be two finite subsets. Then
  \[\isom{\bA \cap \bB}{\bbQU} = \overline{ \langle \isom{\bA}{\bbQU},
    \isom{\bB}{\bbQU} \rangle}.\]
\end{corollary}

\section{Isometries of the Randomly Ordered Urysohn Space}
\label{sec:ordUrysohn}

There is a rich variety of linearly ordered \Fraisse limits, of which
countable dense linear ordering without endpoints is
the simplest example. In fact, as proved in \cite{KPT}, if the group
of automorphisms of a particular \Fraisse class $\cK$ is extremely
amenable, then there is a linear ordering on the \Fraisse limit of
$\cK$ that is preserved by all automorphisms.

We consider another example of a linearly
ordered \Fraisse limit: randomly ordered rational Urysohn space
$\oUr$.

Let us briefly remind
the definition of this structure. Formally speaking, one has to
consider the \Fraisse class $\cM$ of finite ordered metric spaces with
rational distances. Then $\oUr$ is, by definition, the \Fraisse limit of
$\cM$. Intuitively one can think of this structure as a classical
rational Urysohn space 
with a linear ordering on top (such that ordering is
isomorphic to the ordering of the rationals) and such that this
ordering is independent of the metric structure.

Our goal is to prove that every two-dimensional class of topological
similarity in the group of automorphisms of $\oUr$ is meager. We would
like to emphasize that the structure of conjugacy classes in
$\aut{\Q}$ and $\aut{\oUr}$ is substantially different. As was
mentioned earlier there is a generic conjugacy class in $\aut{\Q}$,
while it is not hard to derive from results in \cite{KR}, that each
conjugacy class in $\aut{\oUr}$ is meager.





Recall (see \cite{KR}, Definition 3.3)
\begin{definition}
  Class $\cK$ of finite structures satisfies {\it weak amalgamation
    property\/} (WAP for short) if for every $\bA \in \cK$ there is
  $\bB \in \cK$ and an embedding $e : \bA \ra \bB$ such that for all
  $\bC \in \cK$, $\bD \in \cK$ and all embeddings $i : \bB \ra \bC$,
  $j : \bB \ra \bD$ there are $\bE \in \cK$ and embeddings $k : \bC
  \ra \bE$, $l : \bD \ra \bE$ such that $k \circ i \circ e = l \circ j
  \circ e$, i.e, in the following diagram paths from $\bA$ to $\bE$
  commute (but not necessarily paths from $\bB$ to $\bE$).

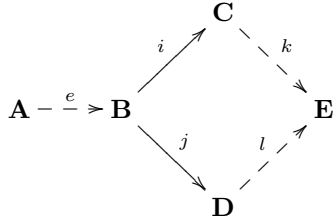
\[\xymatrix{    &     & \bC \ar@{-->}[dr]^{k} &     \\
          \bA \ar@{-->}[r]^{e} & \bB \ar[ur]^{i} \ar[dr]^{j}&     & \bE \\
              &     & \bD \ar@{-->}[ur]^l&     }\]

            Class $\cK$ satisfies {\it local weak amalgamation
              property} if for some $\bA \in \cK$ weak amalgamation
            holds for the class of structures $\bB \in \cK$ that
            extend $\bA$.
\end{definition}

\begin{definition}
  Let $\cK$ be a \Fraisse class. We associate with it a class of
  structures $\cKp$. Elements of $\cKp$ are partial isomorphisms of
  $\cK$, more precisely tuples
  \[(\bA; p : \bA' \to \bA''),\] where $\bA$, $\bA'$ and $\bA'' \in
  \cK$, $\bA'$, $\bA'' \subseteq \bA$ and $p$ is an isomorphisms.
\end{definition}




\begin{theorem}[Kechris--Rosendal, see \cite{KR}, Theorem
  3.7]\label{thm:localWAP}
  Group of automorphisms of a \Fraisse class $\cK$ has a non-meager
  conjugacy class if and only if class $\cKp$ satisfies local weak
  amalgamation property.
\end{theorem}




\begin{proposition}
  Every conjugacy class in $\aut{\oUr}$ is meager.
\end{proposition}

\begin{proof}
  By Theorem \ref{thm:localWAP} it is enough to show that the
  class $\cMp$ has no local WAP. Let $\bar{\bA} = (\bA,\ \phi : \bA'
  \ra \bA'') \in \cMp$, and assume without loss of generality that
  $\phi$ has at least one non-fixed point (otherwise take an extension
  of $\phi$). We claim that the class of structures that extend $\bA$
  does not have WAP.

  Fix $\bar{\bB} = (\bB, \psi : \bB' \ra \bB'')$ that extends $\bA$
  and assume for notational simplicity that $\bA \subseteq \bB$. Let
  $x \in \bA'$ be such that $\phi(x) \ne x$ and let $\orb{\phi}{x}$ be
  the orbit of $x$ under $\phi$.  Then $\orb{\psi}{x} \supseteq
  \orb{\phi}{x}$. Now take (by Lemma \ref{sec:isom-rand-order-diffext}
  and Proposition \ref{sec:isom-to-ord}) two structures $\bar{C} =
  (\bC,\ \sigma: \bC' \ra \bC'') \in \cMp$, $\bar{D} = (\bD,\ \tau:
  \bD' \ra \bD'') \in \cMp$ with embeddings $i : \bar{\bB} \ra
  \bar{\bC}$, $j : \bar{\bB} \ra \bar{\bD}$ and such that $\orb{\sigma
    \circ i}{x}$ and $\orb{\tau \circ j}{x}$ are disjoint over
  $\orb{\phi}{x}$. Then there is no weak amalgamation of $\bar{\bC}$
  and $\bar{\bD}$ over $\bar{\bB}$ and $\bar{\bA}$.
\end{proof}

With classes of topological similarity the situation is rather
different. All non-trivial elements in $\aut{\oUr}$ fall into a single
class of topological similarity. And more generally, if $\mathbb{K}$
is any countable linearly ordered structure and $\aut{\mathbb{K}}$ is
endowed with a topology of pointwise convergence ($\mathbb{K}$ is
discrete here), then
$\aut{\mathbb{K}}$ has exactly two classes of topological similarity
(unless $\aut{\mathbb{K}} = \{id\}$, then, of course, there is only one): all
non-trivial automorphisms generate a discrete copy of $\Z$ and hence fall
into a single class. Thus, in spite of the previous proposition, it
makes sense to ask if there is a non-meager two-dimensional similarity
class in $\aut{\oUr}$.

We define the notions of elementary and piecewise elementary pairs of
partial isomorphisms of $\oUr$ and the notion of liberation exactly like for
the partial isomorphisms of the rationals.

It turns out that the analog of Theorem \ref{thm:rationals} for the
randomly ordered Urysohn space takes place. Let us first briefly
sketch the idea of the proof before diving into the details. We will
prove that, again, for a generic pair there is a sequence of reduced
words, such that this pair converges along it. One can repeat all the
arguments up to Lemma \ref{main-lemma} (only obvious changes are
necessary). So one gets for a piecewise elementary pair $(p,q)$ a
triple $(p',q',w)$ that liberates $p$ in $(p,q)$. But now, opposite to
the case of rationals, one can not in general declare that
$p'(w(p',q')(c)) = c$ for $c \in \essen{p} \cup \essen{q}$, since such
a $p'$ may be not an isometry. At this moment we have to take further
extensions of $p'$ and $q'$. But once an analog of Lemma
\ref{main-lemma} is proved for the Urysohn case, the rest of Theorem
\ref{thm:rationals} goes unchanged.


If $p$ is a partial isometry, we can use amalgamation of its domain
with a one point metric space over an empty set to add a fixed point
for $p$. Using this observation the following two lemmata, which are
analogs of Lemma \ref{liberating-general} and Lemma
\ref{liberating-preserve}, are proved like for the rationals, and we
omit the details.

\begin{lemma}\label{lem:ur-libpres}
  Let $(p,q)$ be a piecewise elementary pair of partial isomorphisms
  of $\oUr$ and assume a triple $(p',q',w)$ liberates $p$ [liberates
  $q$] in $(p,q)$. Let $u = t^nv$ [$u = s^mv$] be a reduced word such
  that $uw$ is irreducible. Then there is a triple $(p'',q'',uw)$ that
  liberates $p$ [liberates $q$] in $(p,q)$. Moreover, one can take
  $p''$ to be an extension of $p'$ and $q''$ to be an extension of
  $q'$.
\end{lemma}

\begin{lemma}\label{lem:ur-libgen}
  Let $(p,q)$ be a piecewise elementary pair of partial isomorphisms
  of $\oUr$ and $u \in F(s,t)$ be a reduced word. Then there is a
  triple $(p',q',vu)$ that liberates $p$ in $(p,q)$ [liberates $q$]
  and such that $|vu| = |v| + |u|$.
\end{lemma}



\begin{lemma}\label{lem:ur-lib-words}
  For any pair $(p,q)$ of partial isomorphisms of the $\oUr$ and any
  word $u \in F(s,t)$ there are extensions $p'$ and $q'$ of $p$ and
  $q$ respectively and a reduced word $w=*u$ such that $w(p',q')(c) =
  c$ for any $c \in \dom{p} \cup \dom{q}$.
\end{lemma}

\begin{proof}
  We can assume that $(p,q)$ is piecewise elementary. By Lemma
  \ref{lem:ur-libgen} there are extensions $\tilde{p}$, $\tilde{q}$ of
  $p$ and $q$ and a reduced word $vu$ such that $(\tilde{p},
  \tilde{q}, vu)$ liberates $p$ in $(p,q)$. Now apply Theorem
  \ref{thm:isomext} (with Remark \ref{sec:main-remark} and Lemma
  \ref{liberating-preserve}) to $\tilde{p}$, $\tilde{q}$ and
  \[\bA = \dom{\bar{p}} \cup \ran{\bar{p}} \cup \dom{\bar{q}} \cup
  \ran{\bar{q}}\] to get extension $\bar{p}$ and $\bar{q}$ and a
  reduced word $v'$. Note that $v'vu$ is reduced, because $v$ starts
  from a power of $t$ and $v'$ by construction ends in a power of
  $s$. By the item \eqref{thm:prop:amal} of Theorem \ref{thm:isomext}
  we can extend $\bar{p}$ to $p'$ by declaring
  \[p'|_{v'vu(\dom{p} \cup \dom{q})} = \id.\] Set $q' = \bar{q}$ and
  $w = (v'uv)^{-1}s(v'uv)$. It is easy to see that $w(p',q')(c) = c$
  holds for any $c \in \dom{p} \cup \dom{q}$.
\end{proof}

\begin{theorem}
  Every two-dimensional class of topological similarity in
  $\aut{\oUr}$ is meager.
\end{theorem}

\begin{proof}
  Repeat the proofs of Lemma \ref{main-lemma} and Theorem
  \ref{thm:rationals} using Lemma \ref{lem:ur-lib-words} instead of
  Lemma \ref{lem:density-words}.
\end{proof}

\begin{remark}
  All the results in this section can be proved for the randomly
  ordered random graph in the same way, as they were proved for the
  randomly ordered rational Urysohn space. One can also formally
  deduce this case from the above results viewing graphs as metric
  spaces with all the distances in $\{0,1,2\}$.
\end{remark}


\begin{thebibliography}{99}
\bibitem{AGW} E. Akin, E. Glasner, B. Weiss, {\it Generically there is
    but one self homeomorphism of the Cantor set},
  Trans. Amer. Math. Soc. 360 (2008), no. 7, 3613--3630.
\bibitem{F} \Fraisse R., {\it Theory of Relations, North-Holland},
  1986.
\bibitem{H} Hodges W., {\it Model Theory}, Cambridge Univ. Press.,
  1993.
\bibitem{KPT} Kechris, A. S., Pestov, V. G., Todorcevic, S., {\it
    Fraïssé limits, Ramsey theory, and topological dynamics of
    automorphism groups}.  Geom. Funct. Anal.  15 (2005), no. 1,
  106--189.
\bibitem{KR} A. S. Kechris, C. Rosendal, {\it Turbulence, amalgamation
    and generic automorphisms of homogeneous structures}, Proc. London
  Math. Soc., 94 (2007), 349--371.
\bibitem{M} J. Melleray, {\it Topology of the isometry group of the
    Urysohn space},  Fund. Math.  207  (2010),  no. 3, 273--287. 
\bibitem{R} C. Rosendal, {\it The generic isometry and measure
    preserving homeomorphism are conjugate to their powers},
  Fund. Math., 205 (2009), no. 1, 1--27.
\bibitem{S} Solecki, S., {\it Extending partial isometries}, Israel
  J. Math. 150 (2005), 315--332.
\bibitem{T} J.K. Truss, {\it On notions of genericity and mutual
    genericity}, Journal of Symbolic Logic, 72 (2007), 755--766.
\bibitem{U1} P.S. Urysohn, {\it Sur un espace m\'etrique universel},
  C. R. Acad. Sci. Paris 180 (1925) 803--806.
\bibitem{U2} P.S. Urysohn, {\it Un espace m\'etrique universel},
  Bull. Sci. Math. 51 (1927), 43--64 and 74--90.
\end{thebibliography}
\end{document}